\documentclass[11pt,a4paper]{amsart}
\usepackage[T1]{fontenc}
\usepackage{lmodern,amsmath,amssymb,amsthm,mathtools}
\usepackage[margin=28mm,headheight=14pt,headsep=16pt]{geometry}
\usepackage{microtype,enumitem,needspace,booktabs,array}
\usepackage[numbers,sort&compress]{natbib}
\usepackage{xcolor}
\usepackage[pdftitle={The RIGP over fields avoiding a curve}]{hyperref}
\setlist[enumerate]{label=\textup{(\roman*)},leftmargin=2.1em,itemsep=3pt,topsep=5pt}
\newtheorem{theorem}{Theorem}[section]
\newtheorem{proposition}[theorem]{Proposition}
\newtheorem{fact}{Fact}
\newtheorem{lemma}[theorem]{Lemma}
\newtheorem{question}[theorem]{Question}
\newtheorem{corollary}[theorem]{Corollary}
\theoremstyle{definition}

\newtheorem{remark}[theorem]{Remark}
\newcommand{\Q}{\mathbb Q}

\DeclareMathOperator{\Gal}{Gal}

\DeclareMathOperator{\Spec}{Spec}
\DeclareMathOperator{\Conf}{Conf}

\numberwithin{equation}{section}
\allowdisplaybreaks[1]
\newcommand{\CXF}{C\mathrm{XF}}

\newcommand{\WXF}{W\mathrm{XF}}
\newcommand{\REG}{\operatorname{REG}}
\newcommand{\CHur}{\operatorname{CHur}}

\newcommand{\ec}{\prec_{\exists}}
\newcommand{\dash}{\dashrightarrow}
\title[The RIGP over fields avoiding a curve]
{The regular inverse Galois problem over fields avoiding a curve
}
\author[Jinhe Ye]{Jinhe Ye}
\email{jinheye@hku.hk}

\begin{document}
\begin{abstract}
Let $k_0$ be a field of characteristic $0$, and let $C/k_0$ be a smooth geometrically integral curve of positive genus with $C(k_0)=\varnothing$.
We prove that the regular inverse Galois problem is solvable over every field existentially closed in the class of $k_0$-fields avoiding $C$, including the curve excluding fields of Johnson and the author.
 To this end, we show that for every finite group $G$, there is a parameter variety over $\mathbb Q$ whose function field admits a regular realization of $G$ without adding points to $C$.
Finally, we establish $\mathrm{RIGP}$ over the variety excluding fields of Szachniewicz and the author, when the excluded variety is a locally closed subvariety of an abelian variety.
\end{abstract}
\maketitle
\tableofcontents
\section{Introduction}

The classical inverse Galois problem asks whether every finite group occurs as the Galois group of a finite Galois extension of $\Q$.
A geometric approach is to construct Galois covers of the projective line and then specialize the realization. For a field $K$ of characteristic $0$, a finite group $G$ is \emph{regularly realizable over $K$} if there is a finite Galois extension $E/K(t)$ with group $G$ such that $K$ is relatively algebraically closed in $E$. The \emph{regular inverse
Galois problem} ($\mathrm{RIGP}$) over $K$ asks whether every finite group is regularly realizable over $K$, and we say \emph{$\mathrm{RIGP}$ over $K$ is solvable} if the answer is yes. By Hilbert's irreducibility theorem, solvability of $\mathrm{RIGP}$ over $\Q$ gives a positive answer to the inverse Galois problem.

A major development came from patching methods, using which Harbater first established regular realizability over fields such as $\Q_p$ and $k((u))$  (see \cite{Harbater2003} for a survey). Subsequently, Pop \cite{pop-embedding} introduced the class of \emph{large fields} and proved that $\mathrm{RIGP}$ is solvable over large fields. A field $K$ is large if every smooth geometrically integral curve over $K$ with one $K$-rational point has infinitely many. Examples include algebraically closed fields, real closed fields, pseudo-algebraically closed (PAC) fields (in the sense of Ax \cite{Ax-fin}), and fields admitting a nontrivial henselian valuation. Large fields admit interesting geometric and arithmetic properties; see~\cite{Pop-little} for a survey on large fields. Pop has shown that the fraction fields of henselian local domains that are not fields are large~\cite{Pophenselian}. Very recently, with Johnson, Tran and Walsberg, we established that the elementary class (in the sense of logic) generated by all fraction fields of henselian local domains that are not fields is exactly the class of large fields~\cite{large->henselian}. This provides an explanation for the shared nice properties of large fields, which include otherwise seemingly different classes of fields.

Regular realizability is preserved under extension of constants, so $\mathrm{RIGP}$ is solvable over every field containing a large subfield. On the other hand, Koenigsmann constructed fields over which $\mathrm{RIGP}$ is solvable, and it contains no large subfield \cite{Koe}. Our main theorem gives another source of such fields, arising from existentially closed fields that avoid a fixed curve. Such fields could also be large depending on the curve and the choice of base field. For the largeness of the existentially closed fields avoiding $C$, see the discussion after Theorem~\ref{thm:main}.

Largeness has a characterization in terms of logic:
\[
K\text{ is large} \quad \textit{if and only if} \quad K\ec K((u)),
\]
where $K\ec K((u))$ denotes existential closedness in the language of rings.
Explicitly, it means that every finite system of polynomial equations and inequations over $K$ that is solvable in $K((u))$ is already solvable in $K$.
Largeness also has strong consequences under model-theoretic tameness assumptions: every infinite large stable field is separably closed \cite{JTWY2024}, and every large simple field is bounded \cite{with-anand}. 
A number of recent papers have further pursued the classification of model-theoretically tame fields and its connection to largeness~\cite{secondpaper, field-top-1, field-top-2, johnson2025largeness, johnson2026translating}.

Fix a field $k_0$ of characteristic $0$ and a smooth geometrically integral curve $C/k_0$ of positive genus with $C(k_0)=\varnothing$. The curve is quasi-projective and its genus is that of its smooth projective completion $\overline C$.
An extension $K/k_0$ \emph{avoids $C$} if $C(K)=\varnothing$. Note that a $k_0$-rational point on $\overline C\setminus C$ ensures that every such field is non-large, since largeness would imply that $\overline{C}(K)$ is infinite.
Such a field is \emph{existentially closed in the class of fields avoiding $C$} if
\[
K\ec F
\qquad\text{for every }F/K
\text{ with }C(F)=\varnothing.
\]
Such fields are canonical examples of fields that are nice model-theoretically, yet do not contain large subfields when $\overline{C}(k_0)\neq\varnothing$ for suitably chosen $k_0$ (e.g. $\mathbb{Q})$. We work in the language of rings with constants for $k_0$, and write $T_{C,k_0}$ for the theory of $k_0$-fields avoiding $C$.

\begin{theorem}\label{thm:main}
Let $k_0$ be a field of characteristic $0$, and let $C/k_0$ be a smooth geometrically integral curve of positive genus with $C(k_0)=\varnothing$.
Then $\mathrm{RIGP}$ is solvable over every field $K/k_0$ that is existentially closed in the class of fields avoiding $C$.
\end{theorem}

Every field avoiding $C$ has such an existentially closed extension \cite[Fact~1.4]{cxf-paper}. One constructs it by successively adjoining witnesses to existential formulas whenever this is possible in an extension still avoiding $C$. 

When $g(C)\geq 2$, Johnson and the author proved that $T_{C,k_0}$ has a model companion $\CXF_{k_0}$, i.e. models of $\CXF_{k_0}$ are precisely these existentially closed fields avoiding $C$ \cite[Theorem~2.1]{cxf-paper}. We obtain the following consequence. Models of $\CXF_{k_0}$ may be large or non-large: a model $K$ is large precisely when $\overline C(K)=\varnothing$ \cite[Theorem~4.1]{cxf-paper}.

\begin{corollary}\label{cor:cxf}
Suppose $g(C)\geq 2$. Then $\mathrm{RIGP}$ is solvable over every model of $\CXF_{k_0}$.
\end{corollary}

We will omit the $k_0$ in the notation when it is clear. Models of $\CXF$ are the prototypical examples of non-large model-complete fields as explained above. This leads to the most general form of Question~\ref{mcfields}, asking if $\mathrm{RIGP}$ is solvable over any model-complete field.

The proof of Theorem~\ref{thm:main} proceeds as follows:
Koenigsmann proved that regular realizability of a fixed finite group is witnessed by a family of existential sentences \cite[Theorem~3.1]{Koe}.
Thus a regular realization over $F$ descends to $K$ whenever $K\ec F$. It suffices to construct an extension of $K$ that admits the desired realization and still avoids $C$.

The main geometric ingredient is Proposition~\ref{prop:parameter}.
For every finite group $G$, we construct a smooth geometrically integral variety $V/\Q$ with a smooth projective compactification $X/\Q$ such that $G\text{ is regularly realizable over }\Q(V),\,
H^0(X,\Omega^1_{X/\Q})=0.$
Fried-V\"olklein's auxiliary group construction and arithmetic
Hurwitz spaces give $V$ \cite{FV1991}. Landesman-Levy's homological stability theorem gives the vanishing of $1$-forms \cite{LL2025}.

Set $F=K(V_K)$. Base change gives a regular realization of $G$ over $F$. The vanishing of $1$-forms rules out nonconstant rational maps from $V_K$ to $C_K$: pullback of a nonzero differential on $\overline C_K$ would provide a contradiction. Since $C(K)=\varnothing$, it follows that $C(F)=\varnothing$. Existential closedness descends the realization to $K$.

Finally, Section~\ref{sec:VXF} treats a class of variety excluding fields from \cite{vxf-paper}.
The same descent argument works, but one must again prevent the parameter variety from mapping nonconstantly to the excluded variety. We do this for subvarieties of abelian varieties.

\textbf{Acknowledgements.}
We are grateful to Jochen Koenigsmann and Benedikt Stock for helpful discussions.

\section{Preliminaries for the reduction}

We gather the results used in the reduction.
Throughout the paper, $k_0$ has characteristic $0$ and $C/k_0$ is a smooth geometrically integral curve of positive genus with $C(k_0)=\varnothing$. Its smooth projective completion is denoted by $\overline C$.
As in the introduction, $T_{C,k_0}$ is the theory of $k_0$-fields avoiding $C$, and a model $K$ is existentially closed if $K\ec F$ for every extension $F/K$ avoiding $C$, i.e. $F\models T_{C,k_0}$.

\begin{lemma}\label{lem:constants}
Let $L\subseteq M$ be fields of characteristic $0$. If $E/L(t)$ is a regular
Galois realization of a finite group $G$ over $L$, then $E_M=E\otimes_{L(t)}M(t)$ is a regular Galois realization of $G$ over $M$.
\end{lemma}

\begin{proof}
Since $E/L$ is regular, $E\otimes_L M$ is a domain. $E_M$ is a localization of this domain and is finite-dimensional over $M(t)$, hence is a field. Base change of the finite-dimensional $L(t)$-vector space $E$ gives $\dim_{M(t)}E_M=[E:L(t)]=|G|$. The automorphisms $\sigma\otimes 1$, for $\sigma\in G$, are distinct $M(t)$-automorphisms of $E_M$. There are exactly $[E_M:M(t)]=[E:L(t)]=|G|$ of them, so this is a Galois extension with group $G$.

Let $\overline M$ be an algebraic closure of $M$. The tensor product $E_M\otimes_M\overline M$ is a localization of $E\otimes_L\overline M$, which is a domain by regularity of $E/L$. Consequently $E_M/M$ is regular.
\end{proof}

\begin{proposition}\label{prop:witness}
For every finite group $G$ there are existential sentences $\REG(G,n)$, $n\in\mathbb N$, in the language of rings such that, for every field $F$,
\[
 G\text{ is regularly realizable over }F
 \quad\Longleftrightarrow\quad
 F\models\REG(G,n)\text{ for some }n.
\]
In particular, if $K\ec F$ and $G$ is regularly realizable over $F$, then $G$ is regularly realizable over $K$. Here $n$ parameterizes the degree of the polynomials defining the realization and its $G$-action.
\end{proposition}

\begin{proof}
This follows from Koenigsmann's theorem \cite[Theorem~3.1]{Koe}.
\end{proof}

\begin{lemma}\label{lem:descent}
Let $K$ be an existentially closed model of $T_{C,k_0}$, and let $F/K$ be an extension with $C(F)=\varnothing$.
Every finite group regularly realizable over $F$ is regularly realizable over $K$.
\end{lemma}

\begin{proof}
By existential closedness and $C(F)=\varnothing$, we have $K\ec F$.
Apply Proposition~\ref{prop:witness}.
\end{proof}

\section{Criterion for avoiding the given curve}

\begin{lemma}\label{lem:no-maps}
Let $k$ be a field of characteristic $0$ and let $V/k$ be a smooth geometrically integral variety. Suppose $V$ has a smooth projective compactification $X$ with $H^0(X,\Omega^1_{X/k})=0$. If $D/k$ is a smooth projective geometrically integral curve of positive genus, then there is no dominant rational map $V\dash D$.
\end{lemma}

\begin{proof}
Suppose such a map $g:V\dash D$ exists. It induces a rational map $X\dash D$. A smooth projective curve of positive genus contains no rational curves over an algebraic closure, so $D$ is geometrically pure; see \cite[Definition~3.1 and Lemmas~3.2, 3.5]{algebraic_hyperbolicity}.
By \cite[Proposition~5.7]{vxf-paper}, we have a morphism $f:X\to D$ extending $g$. In characteristic $0$, pullback along $f$ is injective on regular $1$-forms. Indeed, the inclusion $k(D)\subseteq k(X)$ is separably generated, so the induced map on differentials is injective at the generic point. Hence
\[
 0\ne H^0(D,\Omega^1_{D/k})
 \hookrightarrow H^0(X,\Omega^1_{X/k})=0.
\]
This is a contradiction.
\end{proof}

\begin{corollary}\label{cor:avoidance}
Assume that $g(C)\geq1$. Let $K/k_0$ avoid $C$, and let $V/\Q$ be smooth and geometrically integral.
Suppose a smooth projective compactification $X/\Q$ of $V$ satisfies $H^0(X,\Omega^1_{X/\Q})=0$. Then $C\bigl(K(V_K)\bigr)=\varnothing.$
\end{corollary}

\begin{proof}
 Taking global sections of $\Omega^1$ commutes with extension of the base field, so we have $H^0(X_K,\Omega^1_{X_K/K})=0$.

A point in $C(K(V_K))$ is a rational $K$-map $V_K\dash C_K$. If it is nonconstant, composing with $C_K\hookrightarrow\overline C_K$ contradicts Lemma~\ref{lem:no-maps}.
If it is constant, the image would belong to $C(K)$ by the geometric integrality of $V$, again a contradiction.\end{proof}
\section{Proof of the main theorem and the construction of the parameter space}\label{sec:construction}

The next proposition supplies the desired extension $F/K$ used in the main theorem. 

\begin{proposition}\label{prop:parameter}
For every finite group $G$, there exist a smooth geometrically integral quasi-projective variety $V/\mathbb Q$ and a smooth projective compactification $X/\mathbb Q$ of $V$ such that:
\begin{enumerate}
\item there is a finite Galois extension $E/\mathbb Q(V)(t)$ with
Galois group isomorphic to $G$, in which $\mathbb Q(V)$ is relatively algebraically closed;
\item $H^0(X,\Omega^1_{X/\mathbb Q})=0$.
\end{enumerate}
\end{proposition}

$V$ will be constructed using~\cite{FV1991}, which will give (i).
For (ii), a homological stability theorem in~\cite{LL2025} compares the first cohomology of $V_{\mathbb C}$ with that of a concrete open subset of affine space, which can then be used to show that the first cohomology of a smooth projective compactification of $V$ vanishes. This enables us to deduce the vanishing of $1$-forms; see~\S\ref{sec:cohomological-inputs}.

Before proving Proposition~\ref{prop:parameter}, let's first see how it implies the main theorem.
\begin{proof}[Proof of Theorem~\ref{thm:main}]
Let $C/k_0$ have positive genus, and let $K$ be an existentially closed model of $T_{C,k_0}$. Fix a finite group $G$.
Choose $V/\Q$ as in Proposition~\ref{prop:parameter} and put $F=K(V_K)$. Since $V$ is geometrically integral, this is a regular extension of $K$ containing $\Q(V)$.
Corollary~\ref{cor:avoidance} gives $C(F)=\varnothing$.

The regular $G$-realization over $\Q(V)$ remains a regular $G$-realization over $F$ by Lemma~\ref{lem:constants}.
Since $F$ avoids $C$, existential closedness gives $K\ec F$;
Lemma~\ref{lem:descent} therefore descends the realization to $K$.
Since $G$ was arbitrary, $\mathrm{RIGP}$ is solvable over $K$.
\end{proof}
Moreover, the above proof provides a uniform witness in the following sense.
\begin{corollary}\label{cor:uniform}
For each finite group $G$, there is an integer $n_G$ such that the following holds: Let $k_0$ be a field of characteristic $0$, and  $C/k_0$ be a smooth geometrically integral curve of positive genus with $C(k_0)=\varnothing$. 
Then $K\models\REG(G,n_G)$ for any $K/k_0$ that is existentially closed in the class of fields avoiding $C$.
In particular, $n_G$ is independent of the excluded positive-genus curve and its field of definition.
\end{corollary}

We now return to the proof of Proposition~\ref{prop:parameter}. We first gather some facts about parameter spaces that will be proved later in the section.
All fields we work with in the remaining part of the section are subfields of $\mathbb{C}$. Some of the facts may be well-known; we include them for the sake of completeness. 
\subsection{From a family of covers to a field extension}
Let $F$ be a field of characteristic $0$. A regular Galois extension $E/F(t)$ is given by a finite map
\[
 Y\longrightarrow\mathbb P^1_F,
 \qquad E=F(Y),
\]
where $Y$ is a smooth projective geometrically integral curve.
The elements of the Galois group act as automorphisms of $Y$ over the fixed target $\mathbb P^1_F$. Conversely, such a cover, with a specified action identifying its Galois group with a finite group $\Gamma$, gives a regular $\Gamma$-extension of $F(t)$.
We call it a \emph{$\Gamma$-cover}.

A \emph{family} of $\Gamma$-covers over a variety $V/\mathbb Q$ is a map
\[
 \mathcal Y\longrightarrow\mathbb P^1_{\mathbb Q}\times V
\]
whose fiber at $v\in V$ is a $\Gamma$-cover $Y_v\to\mathbb P^1$.

The generic point of $V$ has residue field $F=\mathbb Q(V)$.
The fiber above the generic point is a cover $Y\to\mathbb P^1_F$.
Its coefficients are rational functions in the parameters, now viewed as elements of $F$, and $t$ still denotes a transcendental element over $F$.
Thus a family over $V$ immediately gives a regular realization over $\mathbb Q(V)$. 

For the proof, we first realize a larger finite group $\Gamma$ with a surjection $\pi:\Gamma\twoheadrightarrow G$. This gives freedom to control our connected parameter space $V$. Once a regular $\Gamma$-extension $E_0/F(t)$ has been constructed, the fixed field $E_0^{\ker\pi}$ gives the required group $G$ via the correspondence (in characteristic 0) between function fields of transcendence degree $1$ and smooth projective curves.

\subsection{The two facts about parameter spaces}

We first explain the objects parametrized by the Hurwitz spaces used below. For these definitions, we work over $\mathbb C$ and fix a finite group $\Gamma$.

\smallskip
\noindent\textbf{Marked covers.}
Given a $\Gamma$-cover $f:Y\longrightarrow\mathbb P^1_{\mathbb C}$ which is unramified above $\infty\in \mathbb{P}^1_\mathbb{C}$, a \emph{marking} is a chosen point $y_\infty\in f^{-1}(\infty)$.

An isomorphism of marked $\Gamma$-covers is an isomorphism of covering curves that commutes with the maps to $\mathbb P^1_{\mathbb C}$ and the specified $\Gamma$-actions and sends one marked point to the other.

\smallskip
\noindent\textbf{Monodromy and inertia classes.}
Let $B\subseteq\mathbb A^1_{\mathbb C}$ be the set of branch points of $f$ (i.e., the image of the ramification points of $f$). The restriction $f^{-1}(\mathbb P^1_{\mathbb C}\setminus B) \to \mathbb P^1_{\mathbb C}\setminus B$ is finite \'etale and a topological covering map with respect to the complex analytic topology.

Fix $b\in B$. Take a loop starting at $\infty$ that goes towards $b$, travels once counterclockwise around a small circle enclosing $b$ and hitting no other branch point, and returns along the same path. Consider lifting this loop to the covering curve starting at the marked point $y_\infty\in f^{-1}(\infty)$.

Since $\Gamma$ acts freely and transitively on $f^{-1}(\infty)$, the endpoint of the lifted loop is uniquely of the form $g_b\cdot y_\infty \text{ for some }g_b\in\Gamma.$

The element $g_b$ is the \emph{local monodromy} associated with these choices. Changing the path or the marked point changes $g_b$ by conjugation. Its conjugacy class is therefore well defined and is called the \emph{inertia class} at $b$.
It is nonidentity since $b$ is a branch point.

Let $c_1,\ldots,c_s$ denote the nonidentity conjugacy classes of $\Gamma$, and take nonnegative integers $n_1,\ldots,n_s$.
We consider marked $\Gamma$-covers with exactly $n_i$ branch points of inertia class $c_i$.

\smallskip
\noindent\textbf{Families of covers and the Hurwitz space.}
A family of marked covers over a complex variety $S$ consists of a finite morphism $f_S:\mathcal Y\longrightarrow\mathbb P^1_{\mathbb C}\times S$, a $\Gamma$-action on $\mathcal{Y}$ over $\mathbb P^1_{\mathbb C}\times S$, and a section above $\infty$
\[
\sigma:S\longrightarrow\mathcal Y,
\qquad
f_S\circ\sigma=(\infty,\operatorname{id}_S).
\]
Moreover, $\mathcal Y\to S$ is a smooth proper family of geometrically connected curves, and the branch divisor is finite \'etale over $S$ and disjoint from $\{\infty\}\times S$. Thus each fiber over $s\in S$, with its action and marking, is a $\Gamma$-cover with the prescribed branching data.

A \emph{fine marked Hurwitz space} for these data is a parameter variety $P$ universal in the following sense:
For every complex variety $S$, pullback gives a natural bijection
\[
\operatorname{Hom}_{\mathbb C}(S,P)
\;\simeq\;
\left\{
\begin{array}{c}
\text{families over $S$ of marked $\Gamma$-covers}\\
\text{with the prescribed branching data}
\end{array}
\right\}\big/\text{isomorphism}.
\]
In particular, points of $P(\mathbb C)$ correspond to isomorphism classes of individual marked covers.
The universal property says more: every algebraic family of such covers is obtained by pulling back the universal family along a unique map to $P$.

Fact~\ref{fact:Hurwitz} says the following: For a suitable auxiliary group $\Gamma$ and sufficiently large multiplicities $n_i=N$, the parameter space and its universal family are defined over $\mathbb Q$, and the parameter space is smooth, affine, and geometrically integral.

\begin{fact}
\label{fact:Hurwitz}
For every nontrivial finite group $G$, there is a finite group $\Gamma$ with $Z(\Gamma)=1$, a surjection $\pi:\Gamma\twoheadrightarrow G$ with the following property:

We write $c_1,\ldots,c_s$ for all the nonidentity conjugacy classes of $\Gamma$. Let $e=\exp(\Gamma)=\operatorname{lcm}\{\operatorname{ord}(g):g\in\Gamma\}$. For every sufficiently large multiple $N$ of $e$, there is a smooth geometrically integral affine variety $P_N/\mathbb Q$ with a universal family of marked $\Gamma$-covers. Each cover is unramified above infinity and has exactly $N$ branch points of each inertia class $c_i$, hence $sN$ branch points in total.
\end{fact}

This follows from Fried-V\"olklein's auxiliary-group construction and arithmetic Hurwitz spaces~\cite{FV1991}; the details of the proof are given in \S\ref{sec:arithmetic-verification}.
In particular, a point of $P_N(F)$ specifies a cover over $F$, its $\Gamma$-action, and an $F$-rational point above infinity.

To state the second fact used, we describe the space that remembers only the branch points and their inertia classes.
Regard each class $c_i$ as a color. A colored configuration with $N$ points of each color can be encoded by monic polynomials
\[
 q_i(T)=T^N+a_{i,1}T^{N-1}+\cdots+a_{i,N},
 \qquad 1\leq i\leq s.
\]
The roots of $q_i$ are the points of color $i$. They must be distinct, and different $q_i$'s have no common root. Consequently, the colored configuration space is
\[
 B_N=\left\{(q_1,\ldots,q_s):
 \begin{array}{l}
 \text{each }q_i\text{ is squarefree},\\
 \text{the }q_i\text{ are pairwise coprime}
 \end{array}\right\}
 \subseteq\mathbb A^{sN}_{\mathbb C}.
\]
The coordinates here are the $sN$ coefficients $a_{i,j}$.
The conditions say that discriminants and pairwise resultants are nonzero, so $B_N$ is a nonempty open subset of affine space.

The \emph{colored branch map}
\[
 b_N:P_{N,\mathbb C}\longrightarrow B_N
\]
forgets the covering curve and the marking, keeping the branch points with their colors. It is an algebraic map. We use this map over $\mathbb C$; the individual colors need not descend to $\mathbb Q$.

\begin{fact}
\label{fact:LL}
For the group $\Gamma$ of Fact~\ref{fact:Hurwitz} and every sufficiently large multiple $N$ of $e$, the colored branch map induces an isomorphism
\begin{equation}\label{eq:stability}
 b_N^*:H^1(B_N,\mathbb Q)
       \xrightarrow{\ \sim\ }H^1(P_{N,\mathbb C},\mathbb Q).
\end{equation}
\end{fact}

Here $H^1$ is singular cohomology of the variety as a complex analytic space. Fact~\ref{fact:LL} roughly says that, once there are enough branch points of every color, remembering only the colored branch configuration is good enough to compute $H^1$ correctly.
It follows from Landesman-Levy's theorem \cite[Theorem~1.4.2]{LL2025}, which we explain in \S\ref{sec:stability-verification}.
\subsection{Vanishing of cohomology implies vanishing of $1$-forms}
\label{sec:cohomological-inputs}

Later we will take $V=P_N$ for some $N$, whereas the desired vanishing condition concerns $H^0(X,\Omega^1_{X})$, where $X$ is its projective compactification. We first explain how this can be achieved.

We use the following facts from Hodge theory.
For a smooth complex variety $T$, the vector space $H^1(T,\mathbb Q)$ has a distinguished subspace $W_1H^1(T,\mathbb Q)$, called its \emph{weight-one part}, with these properties:
\begin{enumerate}
\item If $\overline T$ is a smooth projective compactification,
restriction identifies
\[
 H^1(\overline T,\mathbb Q)\simeq W_1H^1(T,\mathbb Q).
\]
\item If an algebraic map $f:T\to S$ induces an isomorphism on
$H^1$, it induces an isomorphism on the weight-one parts as well.
\end{enumerate}
The first property identifies the part coming from a projective compactification. Deligne proves the description of its image in~\cite[Corollary~3.2.17]{Del};  restriction is injective in degree one because loops in $\overline T$ can be moved off the boundary, which has real codimension at least two. Hence we have an isomorphism as needed.
The second property follows from functoriality and strictness for the weight filtration~\cite[Theorems~2.3.5 and~3.2.5]{Del}.

\begin{lemma}\label{lem:compactification}
Let $U/\mathbb C$ be smooth and connected, and let $B$ be a nonempty open subset of $\mathbb A^r_{\mathbb C}$.
Suppose an algebraic map $b:U\to B$ induces an isomorphism
\[
 b^*:H^1(B,\mathbb Q)\xrightarrow{\ \sim\ }H^1(U,\mathbb Q).
\]
Then every smooth projective compactification $\overline U$ of $U$ satisfies $H^0(\overline U,\Omega^1_{\overline U/\mathbb C})=0$.
\end{lemma}

\begin{proof}
Since $B$ is open in $\mathbb A^r$, it is also open in $\mathbb P^r$. The latter is a smooth projective compactification of $B$ and has $H^1(\mathbb P^r,\mathbb Q)=0$. Property (i) gives $W_1H^1(B,\mathbb Q)=0$.
By property (ii), the isomorphism $b^*$ induced by $b$ gives $W_1H^1(U,\mathbb Q)=0$. Applying property (i) to $U$ yields $H^1(\overline U,\mathbb Q)=0.$ Finally, Hodge decomposition for the smooth projective variety $\overline U$ gives
\[
 0=H^1(\overline U,\mathbb C)
 \simeq H^0(\overline U,\Omega^1_{\overline U/\mathbb C})
       \oplus
       \overline{H^0(\overline U,\Omega^1_{\overline U/\mathbb C})};
\]
see~\cite[\S1.4, Theorem~1.13, Remark~1.16]{VoisinHodge}.
So both summands are zero.
\end{proof}

\subsection{Proof of Proposition~\ref{prop:parameter}}
\label{sec:proof}

\begin{proof}
For the trivial group, take $V=\mathbb A^1_{\mathbb Q}$, $X=\mathbb P^1_{\mathbb Q}$.
Suppose $G$ is nontrivial.

\smallskip
\noindent\emph{Step 1: The parameter variety.}
Choose $\Gamma$, $\pi$, and $e$ as in
Fact~\ref{fact:Hurwitz}. Choose a multiple $N$ of $e$ sufficiently large so that Facts~\ref{fact:Hurwitz} and \ref{fact:LL} apply, and put $V=P_N, F=\mathbb Q(V).$
The variety $V$ is smooth and geometrically integral over $\mathbb Q$. Resolution of singularities provides a smooth projective compactification $X/\mathbb Q$ \cite{Hironaka-resolution}.

\smallskip
\noindent\emph{Step 2: The regular realization.}
The generic point $\Spec F\to V$ gives a $\Gamma$-cover $Y\to\mathbb P^1_F$ from the universal family.
Its function field $E_0=F(Y)$ is Galois over $F(t)$ with group $\Gamma$. Since $Y$ is geometrically integral, $F$ is relatively algebraically closed in $E_0$.
Setting $E=E_0^{\ker\pi},$ we have $\Gal(E/F(t))\simeq G$.
Geometric integrality of $Y$ implies that $E/F(t)$ is regular, proving (i).

\smallskip
\noindent\emph{Step 3: The vanishing of $1$-forms.}
Fact~\ref{fact:LL} gives the algebraic branch map $b_N:V_{\mathbb C}\to B_N$ and its isomorphism on $H^1$.
The space $B_N$ is open in $\mathbb A^{sN}_{\mathbb C}$, where $s$ is the number of nonidentity conjugacy classes of $\Gamma$.
Lemma~\ref{lem:compactification} therefore gives $H^0(X_{\mathbb C},\Omega^1_{X_{\mathbb C}/\mathbb C})=0$.
Global sections commute with extension of the base field, so
\[
 H^0(X,\Omega^1_{X/\mathbb Q})\otimes_{\mathbb Q}\mathbb C
 \simeq H^0(X_{\mathbb C},\Omega^1_{X_{\mathbb C}/\mathbb C})=0.
\]
Hence $H^0(X,\Omega^1_{X/\mathbb Q})=0$, which proves (ii).
\end{proof}

\subsection{Proof of Fact~\ref{fact:Hurwitz}}
\label{sec:arithmetic-verification}

We first isolate the three facts used in the proof.\\
The first one is \cite[\S2.4, Lemma~2]{FV1991}.
\begin{fact}
\label{fact:auxiliary-group}
Every finite group $G$ is a quotient of a finite centerless group $\Gamma$ admitting a Schur cover
\[
1\longrightarrow M\longrightarrow S
 \longrightarrow\Gamma\longrightarrow 1
\]
such that
\[
M=
\bigl\langle [u,v]:
u,v\in S,\ [u,v]\in M\bigr\rangle.
\]
Here a Schur cover means
\[
M\simeq H_2(\Gamma,\mathbb Z),
\qquad M\subseteq [S,S]\cap Z(S).
\]
\end{fact}
The following is the Conway-Parker theorem in  \cite[Appendix]{FV1991}.
\begin{fact}[Conway-Parker transitivity]
\label{fact:braid-transitivity}
Let $\Gamma$ be a finite group admitting a Schur cover as in Fact~\ref{fact:auxiliary-group}. Let $c_1,\ldots,c_s$ denote the nonidentity conjugacy classes of $\Gamma$.

There is an integer $B$ such that, whenever $n_1,\ldots,n_s\geq B$, the tuples
\[
(g_1,\ldots,g_r),\qquad r=n_1+\cdots+n_s,
\]
satisfying
\[
g_1\cdots g_r=1,\qquad
\langle g_1,\ldots,g_r\rangle=\Gamma,
\qquad
\#\{j:g_j\in c_i\}=n_i
\]
form a single braid orbit modulo simultaneous conjugation by $\Gamma$, if nonempty.
The braid action is generated by the Hurwitz moves
\[
(\ldots,a,b,\ldots)
\longmapsto
(\ldots,b,b^{-1}ab,\ldots).
\]
\end{fact}

For the third fact, see \cite[\S\S2.1-2.2]{FV1991} and \cite[Theorem~1.7, \S1.3.1, and Proposition~2.5]{Debes2026}.
\begin{fact}
\label{fact:arithmetic-hurwitz}
Let $\Gamma$ be a finite centerless group, and fix a multiset of $r$ nonidentity conjugacy classes that is invariant under the power maps
$c\longmapsto c^a$, $\gcd(a,|\Gamma|)=1$.

There is a Hurwitz space $H/\mathbb Q$, finite \'etale over $\Conf_r(\mathbb A^1)$, parametrizing geometrically connected $\Gamma$-covers of the fixed projective line with this inertia-class multiset and with all branch points finite.
It is a fine moduli space and carries a universal family.

Its connected components over $\mathbb C$ correspond to braid orbits on generating product-one monodromy tuples modulo simultaneous conjugation.

\end{fact}

Here $\Conf_r(\mathbb A^1)$ denotes the space of unordered sets of $r$ distinct points. It is the principal open subset of $\mathbb A^r$ parametrizing monic squarefree polynomials of degree $r$, and is therefore smooth and affine.

After fixing the branch points and standard loops, the fibers are described by generating product-one monodromy tuples with the prescribed inertia classes.
For unmarked covers, these tuples are taken modulo simultaneous conjugation by $\Gamma$; for marked covers, they are actual tuples.
In both cases, the connected components correspond to orbits of the Hurwitz braid action.

\begin{proof}[Proof of Fact~\ref{fact:Hurwitz}]
Fix a nontrivial finite group $G$. By Fact~\ref{fact:auxiliary-group}, choose a finite centerless group $\Gamma$ and a surjection $\pi:\Gamma\twoheadrightarrow G$ with the stated Schur-cover property.
Write $c_1,\ldots,c_s$ for its nonidentity conjugacy classes, and take $e=\exp(\Gamma)$.

By Fact~\ref{fact:braid-transitivity}, for all sufficiently large $N$, the generating product-one tuples with $N$ entries in each $c_i$, if nonempty, form a single braid orbit modulo simultaneous conjugation.

We check nonemptiness for every sufficiently large multiple $N$ of $e$.
Start with a block of $e$ copies of each nonidentity element of $\Gamma$. Each block has product one, and together the blocks generate $\Gamma$.
There are $e|c_i|$ entries in class $c_i$.
Whenever $N\geq e\max_i|c_i|
\text{ and } e\mid N,$
we can add blocks of $e$ copies of a representative of each $c_i$ until every class occurs exactly $N$ times.
Thus the required tuples form a nonempty single orbit modulo simultaneous conjugation for every sufficiently large multiple $N$ of $e$.

The power maps
\[
c\longmapsto c^a,\qquad \gcd(a,|\Gamma|)=1,
\]
permute the nonidentity conjugacy classes.
Consequently, the multiset containing $N$ copies of every $c_i$ satisfies the rationality hypothesis of
Fact~\ref{fact:arithmetic-hurwitz}.
That fact gives a space $H_N\longrightarrow\Conf_{sN}(\mathbb A^1)$
over $\mathbb Q$, together with a universal family $f:\mathcal Y_N
 \longrightarrow\mathbb P^1_{\mathbb Q}\times H_N$.

The space $H_N$ is smooth and affine because it is finite \'etale over a smooth affine variety.
By the component description in
Fact~\ref{fact:arithmetic-hurwitz}, $H_{N,\mathbb C}$ is nonempty and connected.

Define
\begin{equation}\label{eq:P-universal}
P_N=f^{-1}(\{\infty\}\times H_N).
\end{equation}
A point of $P_N$ specifies a cover parametrized by $H_N$ together with a point on that cover above infinity.
Since every cover is unramified there, the projection $P_N\longrightarrow H_N$
 is finite \'etale of degree $|\Gamma|$.
Thus $P_N$ is smooth and affine over $\mathbb Q$.
Pulling the universal family back to $P_N$, with its tautological point above infinity, gives the universal family of marked covers.

We now show that adding the marking preserves connectedness of the parameter space. The fibers of $P_{N,\mathbb C}$ are described by actual monodromy tuples, whereas those of $H_{N,\mathbb C}$ are described by tuples modulo simultaneous conjugation.
We show that every simultaneous conjugate of a generating product-one tuple already lies in its braid orbit; see \cite[Remark~1.9]{Debes2026} for a comparison between marked and unmarked Hurwitz spaces.

Let $g=(g_1,\ldots,g_r)$ be such a tuple.
Successively moving its first entry to the end by Hurwitz moves gives
\[
L(g)=(g_2,\ldots,g_r,g_1),
\]
where the product-one relation is used.
Successively moving its last entry to the front gives
\[
M(g)=
(g_r,g_r^{-1}g_1g_r,\ldots,g_r^{-1}g_{r-1}g_r).
\]
Hence
\[
L(M(g))=
(g_r^{-1}g_1g_r,\ldots,g_r^{-1}g_rg_r).
\]
This shows that simultaneous conjugation by $g_r^{-1}$ is achieved within the braid orbit.
Since Hurwitz moves commute with simultaneous conjugation and the entries generate $\Gamma$, every simultaneous conjugate of $g$ belongs to its braid orbit.

Thus, the single orbit modulo conjugation is already a single orbit on actual tuples.
It follows that $P_{N,\mathbb C}$ is connected.
Since it is smooth, it is integral, proving that $P_N$ is geometrically integral over $\mathbb Q$.

The universal marked family has geometrically integral covering curves, is unramified above infinity, and has exactly $N$ branch points of each inertia class.
This proves Fact~\ref{fact:Hurwitz}.
\end{proof}

\subsection{Proof of Fact~\ref{fact:LL}}
\label{sec:stability-verification}
We need the following terminology before proving the result.

Let $\Gamma$ be a finite centerless group, put $c=\Gamma\setminus\{1\}$, and write $c_1,\ldots,c_s$ for its nonidentity conjugacy classes.
Fix integers $n_1,\ldots,n_s\geq 0$ with $r=n_1+\cdots+n_s\geq 1$.

Consider the finite set
\[
T_{\mathbf n}=
\left\{
(g_1,\ldots,g_r)\in c^r:
\begin{array}{l}
\langle g_1,\ldots,g_r\rangle=\Gamma,\\
\#\{j:g_j\in c_i\}=n_i
\text{ for every }i
\end{array}
\right\}.
\]
The braid group $\operatorname{Br}_r$ acts on this set by the Hurwitz moves
\[
(\ldots,g_j,g_{j+1},\ldots)
\longmapsto
(\ldots,g_{j+1},g_{j+1}^{-1}g_jg_{j+1},\ldots).
\]
These moves preserve both the generated subgroup and the class multiplicities.

Let $\Conf_r(\mathbb A^1_{\mathbb C})$ be the configuration space of unordered sets of $r$ distinct points.
After choosing a base configuration and the standard identification of its fundamental group with $\operatorname{Br}_r$, the above action determines a finite topological covering of this configuration space.
By the Riemann existence theorem, this covering is the analytification of a finite \'etale algebraic cover.
We denote that algebraic cover by
\[
\CHur^c_{n_1,\ldots,n_s}
\longrightarrow
\Conf_r(\mathbb A^1_{\mathbb C}).
\]
This is the \emph{pointed Hurwitz space on the generating locus}; see \cite[\S2.2]{LL2025} and~\cite[\S1.2]{Debes2026}for more details.

Let $B_{\mathbf n}$ denote the colored configuration space: its points are tuples $(D_1,\ldots,D_s)$ of pairwise disjoint finite subsets of $\mathbb A^1(\mathbb C)$ with $|D_i|=n_i$.
Recording the inertia class of each branch point gives an algebraic colored branch map $b_{\mathbf n}:
\CHur^c_{n_1,\ldots,n_s}\longrightarrow B_{\mathbf n}$. Thus $D_i$ consists of the branch points whose local monodromy belongs to $c_i$.

No product-one condition is imposed in the definition of the pointed Hurwitz space on the generating locus.
Thus the condition $g_1\cdots g_r=1$
defines a union of connected components. Geometrically, this is the locus of covers unramified above infinity, on which the pointing is an ordinary chosen point above infinity.
When $n_1=\cdots=n_s=N$, this locus is identified with $P_{N,\mathbb C}$; see \cite[\S\S2.2-2.3]{LL2025}.
Its connectedness follows from Fact~\ref{fact:Hurwitz}.
We state the specialization of Landesman–Levy's theorem used below.

\begin{fact}
\label{fact:homological-stability}
Let $\Gamma$ be a finite centerless group, write $c_1,\ldots,c_s$ for its nonidentity conjugacy classes, and put $c=\Gamma\setminus\{1\}$.

Let $\CHur^c_{n_1,\ldots,n_s}$ denote the pointed Hurwitz space on the generating locus, and let $B_{n_1,\ldots,n_s}$ denote the configuration space with $n_j$ points of color $c_j$.
There are constants $I,J$, depending only on $\Gamma$, such that, for every $i\geq 0$ and every choice of multiplicities satisfying
\[
n_j>Ii+J\quad\text{for all }j,
\]
the colored branch map on every connected component $Z\subseteq\CHur^c_{n_1,\ldots,n_s}$ induces an isomorphism $H_i\!\left(Z,\mathbb Z[1/|\Gamma|]\right)
\xrightarrow{\ \sim\ }
H_i\!\left(B_{n_1,\ldots,n_s},
           \mathbb Z[1/|\Gamma|]\right).$

\end{fact}

This is \cite[Theorem~1.4.2]{LL2025}, specialized to the conjugation operation on $c$. In their statement, the coefficient ring is $\mathbb Z[1/|G_c^0|]$, where $G_c^0$ is the reduced structure group. For our purposes, $c=\Gamma\setminus\{1\}$ and $G_c^0=\Gamma/Z(\Gamma)=\Gamma$.

\begin{proof}[Proof of Fact~\ref{fact:LL}]
Fix $N$ as in Fact~\ref{fact:Hurwitz}. After fixing the branch points and choosing loops, the fibers of $\CHur^c_{N,\ldots,N}$ are described by generating tuples with $N$ entries in each $c_j$; see \cite[Definitions~2.2.2 and~2.3.1]{LL2025}. As explained before, the product-one tuples form a union of connected components, since Hurwitz moves preserve their product.

Moreover, the product-one locus is precisely $P_{N,\mathbb C}$, as explained above.

This identification is algebraic and preserves the colored branch map. Indeed, the two constructions give finite \'etale covers of the configuration space with the same braid action on their fibers. Moreover, in both constructions, the color of a branch point is its inertia class.

By Fact~\ref{fact:Hurwitz}, $P_{N,\mathbb C}$ is connected. It is therefore a single component of $\CHur^c_{N,\ldots,N}$, and its colored branch map is $b_N:P_{N,\mathbb C}\longrightarrow B_N$.

Choose the lower bound on $N$ large enough that $N>I+J$. Apply Fact~\ref{fact:homological-stability} with
$i=1, n_1=\cdots=n_s=N, Z=P_{N,\mathbb C}$.

It gives an isomorphism
$(b_N)_*: H_1\!\left(P_{N,\mathbb C},\mathbb Z[1/|\Gamma|]\right)\xrightarrow{\ \sim\ }H_1\!\left(B_N,\mathbb Z[1/|\Gamma|]\right)$.

Extending coefficients to $\mathbb Q$ and applying the universal coefficient theorem over $\mathbb Q$, we obtain the dual isomorphism $b_N^*:H^1(B_N,\mathbb Q) \xrightarrow{\ \sim\ } H^1(P_{N,\mathbb C},\mathbb Q)$. This proves Fact~\ref{fact:LL}.
\end{proof}
Before ending the section, note that the fact that the parameter space is defined over $\mathbb{Q}$ gives the following about existentially closed fields simultaneously avoiding multiple curves.
\begin{remark}
Let $k_0$ be a field of characteristic zero, and let $\mathcal{C}$ be any collection of smooth geometrically integral curves of positive genus over $k_0$ such that $C(k_0)=\varnothing$ for every $C\in\mathcal{C}$.
Consider the class of extensions $K/k_0$ with $C(K)=\varnothing \text{  for every }C\in\mathcal{C}$.

Existentially closed models of this class exist, and every such field solves
the regular inverse Galois problem by invoking Proposition~\ref{prop:parameter}. Indeed, fix a finite group $G$, and choose the parameter variety $V$ and its smooth projective compactification $X$ as in Proposition~\ref{prop:parameter}. Put $F=K(V_K)$. $H^0(X_K,\Omega^1_{X_K/K})=0$ gives that every rational map from $X_K$ to the smooth projective completion of a curve $C_K$, with $C\in\mathcal C$ is constant. Thus $F$ remains in the class, and existential closedness descends the regular $G$-realization from $F$ to $K$.

In particular, over $\mathbb Q$ we may take
\[
\mathcal C=
\left\{
D\setminus D(\mathbb Q):
\begin{array}{l}
D/\mathbb Q\text{ is smooth projective geometrically integral},\\
g(D)>0,\quad |D(\mathbb Q)|<\infty
\end{array}
\right\}.
\] 
The resulting existentially closed fields simultaneously preserve the rational points of all these projective curves. Every such field $K$ is a proper regular extension of $\mathbb{Q}$. Nevertheless, $K$ remains very different from $\mathbb{Q}$, since existential closedness forces the equation $x^2+y^2=-1$ to have a solution in $K$. On the other hand, fields that are existentially closed among those avoiding a fixed genus $0$ curve over $\mathbb{Q}$ with no rational points are necessarily large, so the RIGP over them is already solvable by Pop's theorem.
\end{remark}

\section{Comparison with Koenigsmann's construction}

\begin{remark}\label{prop:koenigsmann-not-cxf} Koenigsmann \cite{Koe} constructed fields $K$ that contain no large subfield, and $\mathrm{RIGP}$ is solvable over $K$. When specializing the construction over the base field $\mathbb{Q}$, it will not give a model of $C\mathrm{XF}$. More precisely, let $K/\mathbb{Q}$ be a field obtained by the construction
in the proof of ~\cite[Theorem~2.1]{Koe}, using the function fields with rational places chosen there. Then $K$ is formally real. In particular, $K$ is not a model of any curve excluding field $C\mathrm{XF}$.

In Koenigsmann's proof, the field $K$ is the compositum of a family of independent regular function fields $F_i/\mathbb{Q}$ of transcendence degree one. Each $F_i$ admits a $\mathbb{Q}$-rational place.

By \cite[Corollary~2.2.6]{EnglerPrestel2005}, For any extension $F/k$ admiting a $k$-rational place, every ordering of $k$ extends to an ordering of $F$. Hence, every $F_i$ is orderable. Iterating, we get that any finite compositum of the $F_i$'s above is orderable, and therefore formally real. So an equation of the form $-1=a_1^2+\cdots+a_m^2, a_i\in K$ would already hold in a formally real finite compositum. Therefore $K$ is formally real.

On the other hand, $C\mathrm{XF}$'s are never formally real. Let $L\models C\mathrm{XF}$, where $C$ is a smooth geometrically integral curve whose smooth projective completion $\overline{C}$ has genus at least two and $C(L)=\varnothing$.

Consider the projective curve $D: X^2+Y^2+Z^2=0 \subseteq \mathbb{P}^2_L.$ We claim that $C(L(D))=\varnothing$. Indeed, an $L(D)$-point of $C$ defines a rational $L$-map $f:D\dashrightarrow C$. Constant maps correspond to elements in $C(L)$, which is impossible. 
If $f$ were nonconstant, it would give a nonconstant morphism from a genus-zero curve to a curve of genus at least two. This contradicts the Riemann-Hurwitz formula.

Hence $L(D)/L$ is an extension that still avoids $C$.  On the affine chart $Z\neq 0$, $x=X/Z$ and $y=Y/Z$ satisfy $x^2+y^2=-1.$ So $L(D)\models \exists x\,\exists y\;(x^2+y^2=-1).$ By $L\prec_\exists L(D)$, the same sentence holds in $L$. Thus no $C$XF is formally real.
\end{remark}

\begin{remark}
If a field $K$ in Koenigsmann's construction avoids a chosen genus at least $2$ curve $C/K$, then~\cite[Theorem~4.16]{cxf-paper} gives a regular extension $M/K$ with $M\models C\mathrm{XF}$; every regular Galois realization over $K$ remains regular with the same Galois group after extending constants to $M$, giving a  quick proof that $\mathrm{RIGP}$  holds over such $M$.

\end{remark}

\section{Extensions to variety excluding fields}\label{sec:VXF}

To distinguish the excluded variety from the parameter variety $V$ of Proposition~\ref{prop:parameter}, denote the former by $W$. Fix a characteristic $0$ field $k_0$ and a geometrically integral quasi-projective $W/k_0$ with $W(k_0)=\varnothing$. Let $T_W$ be the theory of extensions $K/k_0$ satisfying $W(K)=\varnothing$. When the model companion exists, $W$ is called \emph{excludable}, and the model companion is denoted $\WXF$. See \cite[Definitions~2.2, 2.5 and 2.7]{vxf-paper} as well.

Szachniewicz and the author prove the existence of $\WXF$ under geometric purity and boundedness hypotheses; in particular, projective Brody hyperbolic varieties over $\Q$ are excludable \cite{vxf-paper}. The following is a consequence of Proposition~\ref{prop:witness}.

\begin{proposition}\label{prop:W-descent}
Let $K$ be existentially closed among extensions of $k_0$ avoiding $W$ and let $F/K$ avoid $W$. Every finite group regularly realizable over $F$ is regularly realizable over $K$.
\end{proposition}

The remaining issue is the following: when using the parameter variety construction and taking $F=K(V_K)$, Proposition~\ref{prop:parameter} is not strong enough to ensure $W(F)=\varnothing$. However, we believe that a generalization of Proposition~\ref{prop:parameter}(ii) could be proved and would be sufficient for this purpose. Instead of pursuing this further, we leave it as a question.

\begin{question}\label{mcfields}
Are the following statements true?
  \begin{enumerate}
 \item  For an excludable variety $W/k_0$, $\mathrm{RIGP}$ is always solvable over models of $W\mathrm{XF}_{k_0}$.
      \item     More generally,  $\mathrm{RIGP}$ is always solvable over a model-complete field.
  \end{enumerate}

\end{question}

We prove RIGP is solvable for the $W$ excluding fields arising from suitable complete intersections in abelian varieties. This partially answers the above question.

\subsection{Subvarieties of abelian varieties}

The geometric fact needed for preserving avoidance of $W$ in this case is the following.

\begin{lemma}\label{lem:abelian-vanishing}
Let $k$ be a field of characteristic $0$, and let $X/k$ be a smooth projective geometrically integral variety such that $H^0(X,\Omega^1_{X/k})=0.$ If $A/k$ is an abelian variety, then every $k$-rational map $f:X\dashrightarrow A$ is constant, with value in $A(k)$.
\end{lemma}

\begin{proof}

By the same argument as before, there is a smooth projective geometrically integral variety $X'$ and a birational morphism $\pi:X'\to X$ such that $f\circ\pi$ extends to a morphism $f':X'\longrightarrow A$. By birational invariance, $H^0(X',\Omega^1_{X'/k})=0$.

Every global $1$-form on $A$ therefore pulls back to zero on $X'$. Since these forms generate $\Omega^1_{A/k}$, the differential $(f')^*\Omega^1_{A/k}\longrightarrow\Omega^1_{X'/k}$ is trivial. In characteristic $0$, a morphism with positive-dimensional image must have a nonzero derivative at the generic point. Hence the image of $f'$ is zero-dimensional.

Because $X'$ is integral and proper, its image is a single closed point $a\in A$. The corresponding residue field embeds into the function field $k\subseteq k(a)\subseteq k(X')$.

The extension $k(a)/k$ is finite. Geometric integrality of $X'$ implies that $k$ is relatively algebraically closed in $k(X')$, so $k(a)=k$. Thus $a\in A(k)$, as required.
\end{proof}

\begin{theorem}\label{thm:rigp-complete-intersections}
Let $k_0$ be a field of characteristic $0$, let $A/k_0$ be an abelian variety, and let $W\subseteq A$ be a geometrically integral locally closed $k_0$-subvariety with $W(k_0)=\varnothing$. Then RIGP is solvable over every field $K/k_0$ that is existentially closed in the class of $k_0$-fields avoiding $W$. In particular, if $W$ is a variety and above and excludable, then RIGP is solvable over every model of $W\mathrm{XF}$.
\end{theorem}

\begin{proof}
Fix such a field $K$ and a finite group $G$. By Proposition~\ref{prop:parameter}, there are a smooth geometrically integral quasi-projective variety $T/\mathbb{Q}$ and a smooth projective compactification $X/\mathbb{Q}$ of $T$ such that $G$ is regularly realizable over $\mathbb{Q}(T)$ and $H^0(X,\Omega^1_{X/\mathbb{Q}})=0$.

Set $F=K(T_K)$. Since $T$ is geometrically integral, $F/K$ is regular and $F$ contains $\mathbb{Q}(T)$. By Lemma~\ref{lem:constants}, $G$ is regularly realizable over $F$. Moreover, base change gives
\[
    H^0(X_K,\Omega^1_{X_K/K})\cong H^0(X,\Omega^1_{X/\mathbb{Q}}) \otimes_{\mathbb{Q}}K=0.
\]

We claim that $W(F)=\varnothing$. Otherwise, an $F$-point of $W$ determines a rational $K$-map $T_K\dashrightarrow W_K$. Composing with the inclusion $W_K\hookrightarrow A_K$ and using the density of $T_K$ in $X_K$ gives a rational map $X_K\dashrightarrow A_K$.

By Lemma~\ref{lem:abelian-vanishing}, this map is constant with value at some $a\in A(K)$. Its generic value lies in $W_K$, so $a\in W(K)$, contrary to $W(K)=\varnothing$. This proves the claim.

Since $F/K$ still avoids $W$, existential closedness gives $K\prec_{\exists}F$. Proposition~\ref{prop:witness} therefore descends the regular $G$-realization to $K$. As $G$ was arbitrary, RIGP is solvable over $K$. The final assertion follows because the models of $W\mathrm{XF}$ are precisely the existentially closed $k_0$-fields avoiding $W$.
\end{proof}

\begin{remark}
Let $\dim A=g\geq 2$, fix a very ample line bundle $\mathcal L$ on $A$, and choose an integer $\left\lceil\frac g2\right\rceil\leq c<g$. For sufficiently large $d$, a general complete intersection
\[
Y=H_1\cap\cdots\cap H_c, \quad  H_i\in|\mathcal L^{\otimes d}|,
\]
is smooth and geometrically integral and has ample cotangent bundle. For ampleness, apply \cite[Theorem~0.1]{amplecotangent} with the fixed multidegree vector $(1,\ldots,1)$; smoothness and geometric integrality follow from Bertini's theorem.
The theorem applies whenever a proper closed $D\subsetneq Y$ is chosen with $W=Y\setminus D,
W(k_0)=\varnothing$. Moreover, such $W$'s are excludable by \cite[Theorem 6.5 and Corollary 6.7]{vxf-paper}, i.e., $W\mathrm{XF}$ exists.

\end{remark}
\noindent
\textbf{Use of AI.}
 The idea behind Proposition~\ref{prop:parameter} arose from discussions with ChatGPT while the author was learning about Hurwitz spaces from~\cite{Debes2026}. ChatGPT also helped check mathematical arguments and locate references for this part of the paper.  ChatGPT and Google Gemini were used to improve the clarity and grammar of the writing.

\renewcommand{\bibfont}{\small}
\bibliographystyle{alpha}
\bibliography{ref}

\end{document}